\documentclass[12pt, reqno]{amsart}
\usepackage{amsmath, amsthm, amscd, amsfonts, amssymb, graphicx, color}
\usepackage[bookmarksnumbered, colorlinks, plainpages]{hyperref}
\hypersetup{colorlinks=true,linkcolor=red, anchorcolor=green, citecolor=cyan, urlcolor=red, filecolor=magenta, pdftoolbar=true}

\newtheorem{theorem}{Theorem}[section]
\newtheorem{lemma}[theorem]{Lemma}

\theoremstyle{definition}
\newtheorem{definition}[theorem]{Definition}

\numberwithin{equation}{section}

\begin{document}

\setcounter{page}{1}

\title[Dual continuous $K$-Frames in Hilbert spaces]{Dual continuous $K$-Frames in Hilbert spaces}

\author[Mohamed Rossafi, Brahim Moalige, Hamid Faraj, Abdeslam Touri and Samir Kabbaj]{Mohamed Rossafi$^1$$^{\ast}$, Brahim Moalige$^2$, Hamid Faraj$^2$, Abdeslam Touri$^2$ \MakeLowercase{and} Samir Kabbaj$^2$}

\address{$^{1}$Department of Mathematics, Faculty of Sciences Dhar El Mahraz, University Sidi Mohamed Ben Abdellah, Fes, Morocco}
\email{\textcolor[rgb]{0.00,0.00,0.84}{rossafimohamed@gmail.com}}

\address{$^{2}$Department of Mathematics, University of Ibn Tofail, B.P. 133, Kenitra, Morocco}
\email{\textcolor[rgb]{0.00,0.00,0.84}{bmoalige@hotmail.com}}
\email{\textcolor[rgb]{0.00,0.00,0.84}{farajham19@gmail.com}}
\email{\textcolor[rgb]{0.00,0.00,0.84}{touri.abdo68@gmail.com}}
\email{\textcolor[rgb]{0.00,0.00,0.84}{samkabbaj@yahoo.fr}}



\date{Received:; 
\newline \indent $^{*}$Corresponding author}

\begin{abstract}
Frame theory is recently an active research area in mathematics, computer science and engineering with many exciting applications in a variety of different fields. This theory has been generalized rapidly and various generalizations of frames in Hilbert spaces. In this papers we study the notion of dual continuous $K$-frames in Hilbert spaces. Also we etablish some new properties.  
\end{abstract} \maketitle

\section{Introduction}
A frame is a set of vectors in a Hilbert space that can be used to reconstruct each vector in the space from its inner products with the frame vectors. These inner products are generally called the frame coefficients of the vector. But unlike an
orthonormal basis each vector may have infinitely many different representations in terms of its frame coefficients.

Introduced by Duffin and Schaeffer in 1952 \cite{1} to study some deep problems in nonharmonic Fourier series, theory of frame in Hilbert space has grown rapidly. After the fundamental paper \cite{2} by Daubechies, Grossman and Meyer, frames
theory began to be widely used, particularly in the more specialized context of wavelet frames and Gabor frames \cite{3}. 

A discrete frame in a separable Hilbert space $\mathcal{H}$ is a sequence $\{f_{i}\}_{i\in I}$  for which there exist positive constants $A, B > 0$ called frame bounds such that
\begin{equation*}
	A\|x\|^{2}\leq\sum_{i\in I}|\langle x,f_{i}\rangle|^{2}\leq B\|x\|^{2}, \;\forall x\in\mathcal{H}.
\end{equation*}

The continuous frames has been defined by Ali, Antoine and Gazeau \cite{4}, called these kinds frames, frames associated with measurable space. For more details, the reader can refer \cite{5}. The concept of continuous $K$-frame in Hilbert space have been introduced in \cite{c-k-frame}.

Many generalizations of the concept of frame have been defined in Hilbert Spaces and Hilbert $C^{\ast}$-modules \cite{r1, r11, r2, r3, r4, r5, r6, r7, r10, r8, r9}.

In this papers, we characterize the concept of dual continuous $K$-frames in Hilbert spaces and we give some new properties.
\section{Preliminaries}
Let $X$ be a Banach space, $(\Omega,\mu)$ a measure space, and function $f:\Omega\to X$ a measurable function. Integral of the Banach-valued function $f$ has defined Bochner and others. Most properties of this integral are similar to those of the integral of real-valued functions. 

Let $(\Omega,\mu)$ be a measure space, let $H$ and $K$ be two separables Hilbert Spaces, we denote $B(H,K)$ the collection of all bounded linear operators from $H$ to $K$, as well $B(H,H)$ is abbreaviated to $B(H)$.\\
For $T \in B(H,K)$, we use the notation $\mathcal{R}(T)$ and $\mathcal{N}(T)$ to denote respectively the range and the null space of $T$.

\begin{definition} \cite{5} \label{d1}
	Let $H$ be a complex Hilbert space, and $(\Omega,\mu)$ be a measure space with positive measure $\mu$.\\
	A map $ F : \Omega \longrightarrow H$ is called a continuous frame with respect to $(\Omega,\mu)$ if :
	\begin{itemize}
		\item [1]- $F$ is weakly measurable, ie: $ \forall f\in H, \quad w \longrightarrow \langle f,F(\omega)\rangle$ is a measurable fonction on $\Omega$.
		\item [2]- There exists two constants $A,B > 0$ such that :
		\begin{equation}\label{d1eq1}
		A\|f\|^{2}\leq \int_{\Omega}|\langle f,F(\omega)\rangle |^{2}d\mu(\omega)\leq B\|f\|^{2} \quad \forall f\in H.
		\end{equation}
	\end{itemize}
\end{definition}
Let $F$ be a continuous frame, the analysis operator $T$ is defined by :
\begin{align*}
T: &H \longrightarrow L^{2}(\Omega)\\
& f \longrightarrow \{\langle f,F(\omega)\rangle \}_{\omega \in \Omega}
\end{align*}
The adjoint operator of $T$, called synthesis operator, is defined by :
\begin{align*}
T^{\ast}: &L^{2}(\Omega) \longrightarrow H\\
& x \longrightarrow \int_{\Omega}x(\omega)F(\omega)d\mu(\omega)
\end{align*}
The frame operator of the continuous frame $F$ is defined by : $S=T^{\ast}T$ such that, is bounded and invertible.\\
Recall that a continuous Bessel sequence $G$ is a dual continuous frame of $F$ if :
\begin{equation*}
f=\int_{\Omega}\langle f,G(\omega)\rangle F(\omega)d\mu (\omega) \qquad \forall f\in H
\end{equation*}
We have : 
\begin{equation*}
f=\int_{\Omega}\langle f,S^{-1}_{F}F(\omega)\rangle F(\omega)d\mu (\omega) \qquad \forall f\in H
\end{equation*}
This show that $S^{-1}_{F}F$ is a dual continuous frame of $F$, called the canonical dual continuous frame of $F$.
\begin{definition}\cite{c-k-frame} \label{d2}
Let $K\in B(H)$, a map $F:\Omega \longrightarrow H$ is said to be a continuous K-frame, if there exists a constants $0<A<B<\infty$ such that :
\begin{equation*}
A\|K^{\ast}f\|^{2}\leq \int_{\Omega}|\langle f,F(\omega)\rangle |^{2}d\mu(\omega)\leq B\|f\|^{2} \quad \forall f\in H.
\end{equation*}
The constants $A$ and $B$ are called the lower and upper continuous K-frame bounds.\\
If $A=B$, $F$ is called a  tight continuous K-frame.\\
If $A=B=1$, $F$ is called the parseval continuous K-frame.\\
If \eqref{d1eq1} holds right, $F$ is called continuous K-Bessel sequence. 
\end{definition}
\begin{definition}\cite{Tousi}
	A Bessel mapping $F$ is said to be $L^{2}$-independent if $\int_{\Omega}\phi(\omega)F(\omega)d\mu(\omega)=0$ for $\phi\in L^{2}(\Omega)$, implies that $\phi=0$ a. e.
\end{definition}
\begin{lemma}\cite{10}\label{l1}
Let $\Lambda \in B(H,K)$ has a closed range, then there exists a unique operator $\Lambda^{\dagger}\in B(K,H)$, called the pseudo-inverse of $\Lambda$, satisfying : 
\begin{align*}
&\Lambda\Lambda^{\dagger}\Lambda=\Lambda \qquad (\Lambda\Lambda^{\dagger})^{\ast}=\Lambda\Lambda^{\dagger} \\
&\Lambda^{\dagger}\Lambda\Lambda^{\dagger}=\Lambda^{\dagger} \qquad (\Lambda^{\dagger}\Lambda)^{\ast}=\Lambda^{\dagger}\Lambda \qquad (\Lambda^{\ast})^{\dagger}=(\Lambda^{\dagger})^{\ast}\\
&\mathcal{N}(\Lambda^{\dagger})=(\mathcal{R}(\Lambda))^{\perp} \qquad \mathcal{R}(\Lambda^{\dagger})=(\mathcal{N}(\Lambda))^{\perp}
\end{align*}
\end{lemma}
\begin{lemma}\cite{11}\label{l2}
	Let $H, H_{1}$ and $H_{2}$ be three Hilbert Spaces, also let $S\in B(H_{1},H)$ and $T\in B(H_{2},H)$. The following statements are equivalents :
	\begin{itemize}
		\item [1 -] $\mathcal{R}(S) \subset \mathcal{R}(T)$.
		\item [2 -] There exist $\lambda > 0 $ such that $SS^{\ast} \leq \lambda TT^{\ast}$.
		\item [3 -] There exist $\theta \in B(H_{1},H_{2})$ such that $S=T\theta$.
	\end{itemize}
	Moreover, if $(1)$, $(2)$ and $(3)$ are valide, then there exist a unique operator $\theta$ such that :
	\begin{itemize}
\item	[a -] $\|\theta\|^{2}=inf\{\mu : SS^{\ast}\leq \mu TT^{\ast}\}$.
\item	[b -] $\mathcal{N}(S)=\mathcal{N}(\theta)$.
\item	[c -] $\mathcal{R}(\theta)\subset \overline{\mathcal{R}(T^{\ast})}$.
	\end{itemize}
\end{lemma}

\section{Main result}
Before giving our main results, we will first demonstrate the following lemmas.
\begin{lemma}\label{l3}
	Let $K\in B(H)$ and $F$ be a continuous Bessel sequence of $H$ with analysis operator $T$. Then $F$ is a continuous $K$-frame of $H$ if and only if:
	
	 $\mathcal{R}(K) \subset \mathcal{R}(T^{\ast})$.
\end{lemma}
\begin{proof}
	It is an immediate consequence of lemma \ref{l2}.
\end{proof}
\begin{lemma}\label{l4}
Suppose that $K\in B(H)$ has closed range and $F$ is a parseval continuous $K$-frame of $H$, then $K^{\dagger}F$ is a dual continuous $K$-Bessel sequence of $F$.
\end{lemma}
\begin{proof}
$F$ is a parseval continuous $K$-frame of $H$, then :
\begin{equation}\label{1}
\|K^{\ast}f\|^{2}=\int_{\Omega}|\langle f,F(\omega)\rangle |^{2}d\mu(\omega) \qquad \forall f\in H.
\end{equation}
Let $g\in {\mathcal{R}(K^{\ast})}$, we have : $g=K^{\ast}(K^{\ast})^{\dagger}g=K^{\ast}(K^{\dagger})^{\ast}g$.  Replace $f$ by $(K^{\ast})^{\dagger}g$ in \eqref{1}, then :
\begin{equation*}
\|K^{\ast}(K^{\ast})^{\dagger}g\|^{2}=\int_{\Omega}|\langle (K^{\ast})^{\dagger}g, F(\omega)\rangle |^{2}d\mu(\omega)
\end{equation*}
so, 
\begin{equation*}
\|g\|^{2}=\int_{\Omega}|\langle g,K^{\dagger}F(\omega)\rangle |^{2}d\mu(\omega).
\end{equation*}
Hence, $K^{\dagger}F$ is a continuous Bessel sequence.

Since $F$ is a parseval continuous $K$-frame, one has
\begin{equation*}
KK^{\ast}f=\int_{\Omega}\langle f,F(\omega)\rangle F(\omega)d\mu(\omega).
\end{equation*}
Then we have : $g=K^{\ast}(K^{\ast})^{\dagger}g=K^{\ast}(K^{\dagger})^{\ast}g$,
\begin{align*}
Kg&=KK^{\ast}(K^{\dagger})^{\ast}g\\
&=\int_{\Omega}\langle (K^{\dagger})^{\ast}g,F(\omega)\rangle F(\omega)d\mu(\omega)\\
&=\int_{\Omega}\langle g,K^{\dagger}F(\omega)\rangle F(\omega)d\mu(\omega).
\end{align*}
If $h\in (\mathcal{R}(K^{\ast}))^{\perp}=\mathcal{N}(K)$\\
lemma \ref{l1} $\Longrightarrow h\in (\mathcal{N}((K^{\ast})^{\dagger})=\mathcal{N}((K^{\dagger})^{\ast})$,\\
then :
\begin{equation*}
\int_{\Omega}\langle h,K^{\dagger}F(\omega)\rangle F(\omega)d\mu(\omega)=\int_{\Omega}\langle (K^{\dagger})^{\ast}h,F(\omega)\rangle F(\omega)d\mu(\omega)=0=Kh.
\end{equation*}
So, for all $f\in H$, we have :
\begin{equation*}
Kf=\int_{\Omega}\langle g,K^{\dagger}F(\omega)\rangle F(\omega)d\mu(\omega).
\end{equation*}
\end{proof}
\begin{lemma}\label{l5}
Suppose that $K\in B(H)$ has closed range and $F$ is a parseval continuous $K$-frame of $H$ with analysis operator $T$, then $G$ is a dual continuous $K$-Bessel sequence of $F$ if and only if there exists $\varphi \in B(H,L^{2}(\Omega))$ such that : $T^{\ast}\varphi = 0$ and $(\varphi f)_{\omega}=\langle f, G(\omega)-K^{\dagger}F(\omega)\rangle \quad \forall f\in H, \quad \forall \omega \in \Omega$.
\end{lemma}
\begin{proof}
Let $G$ be a dual continuous $K$-Bessel sequence of $F$, \begin{align*}
\varphi : &H\longrightarrow L^{2}(\Omega)\\
&f\longrightarrow \varphi f
\end{align*}
wich is defined by: $(\varphi f)_{\omega}=\langle f,G(\omega)-K^{\dagger}F(\omega)\rangle$.
One has
\begin{align*}
T^{\ast}(\varphi f)
&=\int_{\Omega}\langle f,G(\omega)-K^{\dagger}F(\omega)\rangle F(\omega) d\mu(\omega)\\
&=\int_{\Omega}\langle f,G(\omega)\rangle F(\omega) d\mu(\omega) - \int_{\Omega}\langle f,K^{\dagger}F(\omega)\rangle F(\omega) d\mu(\omega)\\
&=Kf-Kf=0.
\end{align*}
Conversely, suppose that exist $\varphi \in B(H,L^{2}(\Omega))$ such that : $T^{\ast}\varphi = 0$ and 
\begin{equation*}	
(\varphi f)_{\omega}=\langle f,G(\omega)-K^{\dagger}F(\omega)\rangle \quad \forall f\in H, \quad \forall \omega \in \Omega.
\end{equation*}
We have :
\begin{align*}
\int_{\Omega}\langle f,G(\omega)\rangle F(\omega)d\mu (\omega)&=\int_{\Omega}\langle f, K^{\dagger}F(\omega)\rangle F(\omega)d\mu(\omega)\\
&=Kf.
\end{align*}
\end{proof}
\begin{theorem}\label{l6}
Suppose that $K\in B(H)$ has closed range and $F$ is a parseval continuous $K$-frame of $H$ with analysis operator $T_{F}$, then $K^{\dagger}F$ is the canonical dual continuous $K$-Bessel sequence of $F$.
\end{theorem}
\begin{proof}
From lemma \ref{l4}, we know that $K^{\dagger}F$ is a dual continuous $K$-Bessel sequence of $F$. To complete the proof, it only needs to prove: $\|T_{F^{'}}\|\leq \|T_{G}\|$ for any dual continuous $K$-Bessel sequence $G$ of $F$, where $T_{F^{'}}$ is the analysis operator of $K^{\dagger}F$.\\
By lemma \ref{l5}, there exist $\varphi \in B(H,L^{2}(\Omega))$ such that $T^{\ast}_{F}\varphi= 0$ and 
\begin{equation*}	
(\varphi f)_{\omega}=\langle f,G(\omega)-K^{\dagger}F(\omega)\rangle \quad \forall f\in H, \quad \forall \omega \in \Omega.
\end{equation*}
On the other hand : $T_{G}^{\ast}=T_{F^{'}}^{\ast} + \varphi$,
\begin{align*}
\|T_{G}^{\ast}f\|^{2}&=\langle T_{G}^{\ast}f,T_{G}^{\ast}f\rangle\\
&=\langle T_{F^{'}}^{\ast}f + \varphi f,T_{F^{'}}^{\ast}f + \varphi f\rangle \\
&=\|T_{F^{'}}^{\ast}f\|^{2} + \langle T_{F^{'}}^{\ast}f,\varphi f\rangle + \langle \varphi f, T_{F^{'}}^{\ast}f \rangle + \|\varphi f\|^{2}\\
&=\|T_{F^{'}}^{\ast}f\|^{2}  +  \|\varphi f\|^{2}\geq \|T_{F^{'}}^{\ast}f\|^{2}
\end{align*}
Hence, $\|T_{F^{'}}\|\leq \|T_{G}\|$ as desired.
\end{proof}
\begin{lemma}
\begin{itemize}
\item [1 -] The canonical continuous dual $K$-Bessel sequence of a parseval continuous $K$-frame $F$, wich will be denoted by $\tilde{F}$ later, is actually a parseval continuous frame on $(\mathcal{N}(K))^{\perp}$.
\item [2 -] The canonical dual continuous $K$-Bessel sequence of parseval continuous $K$-frame $F$ is precisely a parseval continuous $K^{\dagger}K$-frame. But in general it is not a parseval continuous $K$-frame. It can naturally generate a new one in the form $K\tilde{F}$.


\end{itemize}
\end{lemma}
\begin{proof}
\begin{itemize}
	\item [1]-\begin{equation*}
	\int_{\Omega}|\langle f,\tilde{F}(\omega)\rangle|^{2}d\mu (\omega)=\|K^{\ast}(K^{\dagger})^{\ast}f\|^{2}=\|(K^{\dagger}K)^{\ast}f\|^{2}=\|K^{\dagger}Kf\|^{2}=\|f\|^{2}.
	\end{equation*}
	\item[2 -] \begin{align*}
	\int_{\Omega}|\langle f,\tilde{F}(\omega)\rangle|^{2}d\mu (\omega)&=\int_{\Omega}|\langle f,K^{\dagger}F(\omega)\rangle|^{2}d\mu (\omega)\\
	&=\|K^{\ast}(K^{\dagger})^{\ast}f\|^{2}\\
	&=\|(K^{\dagger}K)^{\ast}f\|^{2} \qquad \forall f\in H,
	\end{align*}
	\begin{align*}
	\int_{\Omega}|\langle f,K\tilde{F}(\omega)\rangle|^{2}d\mu (\omega)&=\int_{\Omega}|\langle f,KK^{\dagger}F(\omega)\rangle|^{2}d\mu (\omega)\\
	&=\int_{\Omega}|\langle (KK^{\dagger})^{\ast}f,F(\omega)\rangle|^{2}d\mu (\omega)\\
	&=\|K^{\ast}(KK^{\dagger})^{\ast}f\|^{2}=\|(KK^{\dagger}K)^{\ast}f\|^{2}\\
	&=\|K^{\ast}f\|^{2} \qquad \forall f\in H.
	\end{align*}
	
\end{itemize}
\end{proof}
\begin{theorem}
Suppose that $K\in B(H)$ has closed range and $F$ is a parseval continuous $K$-frame of $H$ with a dual continuous $K$-Bessel sequence $G$. Then $G$ is the canonical dual continuous $K$-Bessel sequence of $F$ if and only if $T^{\ast}_{G}T_{G}=T^{\ast}_{G}T_{H}$ for any dual continuous $K$-Bessel sequence $H$ of $F$, where $T_{G}$ and $T_{H}$ denote the analysis operators of $G$ and $H$ respectively. 
\end{theorem}
\begin{proof}
Let us first assume that $G=\tilde{F}$.\\
If we denote by $T_{F}$ the analysis operator of $F$ then a direct calculation can show that $T_{G}=T_{F}(K^{\dagger})^{\ast}$.\\
From this fact and taking into account the fact that :
\begin{equation*}
T^{\ast}_{F}(T_{G}f-T_{H}f)=\int_{\Omega}\langle f,\tilde{F}(\omega)\rangle F(\omega) d\mu (\omega)-\int_{\Omega}\langle f,H(\omega)\rangle F(\omega)d\mu (\omega)=0.
\end{equation*}
We obtain for any $f,g \in H$:
\begin{equation*}
\langle (T_{G}-T_{H})f,T_{G}g\rangle = \langle (T_{G}-T_{H})f,T_{F}(K^{\dagger})^{\ast}g\rangle=\langle K^{\dagger}T^{\ast}_{F}(T_{G}-T_{H})f,g\rangle = 0.
\end{equation*}
Thus $T^{\ast}_{G}(T_{G}f-T_{H}f)=0$ then $T^{\ast}_{G}T_{G}=T^{\ast}_{G}T_{H}$.\\
For the converse, suppose that $T^{\ast}_{G}T_{G}=T^{\ast}_{G}T_{H}$, for any dual continuous $K$-Bessel sequence $H$ of $F$. Then :
\begin{equation*}
\|T_{G}\|^{2}=\|T^{\ast}_{G}T_{G}\|=\|T^{\ast}_{G}T_{H}\|\leq \|T_{G}\|\|T_{H}\|.
\end{equation*}
So, $\|T_{G}\|\leq \|T_{H}\|$ implying that $G$ is the canonical continous $K$-Bessel sequence of $F$.
\end{proof}
Now it is legitimate to pose the following question: Under what condition will a parseval continuous $K$-frame admit a unique dual continuous $K$-Bessel sequence?
\begin{theorem}\label{t1}
Suppose that $K\in B(H)$ has closed range and $F$ is a parseval continuous $K$-frame of $H$ with analysis operator $T_{F}$. Then $F$ has a unique dual continuous $K$-Bessel sequence if and only if $\mathcal{R}(T_{F})=L^{2}(\Omega)$.
\end{theorem}
\begin{proof}
Suppose that $\mathcal{R}(T_{F})=L^{2}(\Omega)$, then $T^{\ast}_{F}$ is injective. Let $G$ and $Q$ be two dual continuous $K$-Bessel sequences of $F$. Then: $\{\langle f, G(\omega)-Q(\omega)\rangle \}_{\omega \in \Omega} \in L^{2}(\Omega)$ and that :
\begin{align*}
0=Kf-Kf&=\int_{\Omega}\langle f,G(\omega)\rangle F(\omega) d\mu (\omega) - \int_{\Omega}\langle f,Q(\omega)\rangle F(\omega) d\mu (\omega)\\
&=\int_{\Omega}\langle f,G(\omega)-Q(\omega)\rangle F(\omega) d\mu (\omega) \\
&=T^{\ast}_{F}(\{\langle f, G(\omega)-Q(\omega)\rangle \}_{\omega \in \Omega}).
\end{align*}
Since, $T^{\ast}_{F}$ is injective, we have $$\langle f,G(\omega)-Q(\omega)\rangle = 0 \quad \forall \omega \in \Omega \quad and \quad \forall f\in H, $$
hence $$G(\omega)=Q(\omega) \quad \forall \omega \in \Omega,$$ so $G=Q$.\\
Conversely, assume contrarity that $\mathcal{R}(T_{F})\neq L^{2}(\Omega)$.\\
Since $F$ is a parseval continuous $K$-frame, it is easely seen that $KK^{\ast}=T^{\ast}_{F}T_{F}$.\\
Hence, $\mathcal{R}(T^{\ast}_{F})=\mathcal{R}(K)$, by lemma \ref{l2}, and $T^{\ast}_{F}$ has closed range as a consequence.\\
Let $S\in B(H,L^{2}(\Omega)) $ be an invertible operator and $0\neq \alpha \in (\mathcal{R}(T_{F}))^{\perp}$.\\
Taking $h=S^{-1}(\alpha)$ and $G(\omega)=\bar{\alpha(\omega)}h$,
for each $\omega \in \Omega$, for every $f\in H$, we have:
\begin{align*}
\int_{\Omega}|\langle f,G(\omega)\rangle|^{2}d\mu (\omega)&=\int_{\Omega}|\langle f,\bar{\alpha(\omega)}h\rangle|^{2}d\mu (\omega)\\
&=\int_{\Omega}|\langle f,h\rangle |^{2}|\alpha(\omega)|^{2}d\mu (\omega)\\
&=|\langle f,h\rangle |^{2}\|\alpha \|_{2}^{2}\\
&\leq \|\alpha \|_{2}^{2} \|h\|^{2} \|f\|^{2}.
\end{align*}
Hence, $G$ is a continuous $K$-Bessel sequence for $H$.\\
Now, let $Q(\omega)=\tilde{F}(\omega) + G(\omega)$ for every $\omega \in \Omega$, then it is easely seen that $Q$ is a continuous $K$-Bessel sequence for $H$.\\
Since $\alpha $ is orthogonal to $\mathcal{R}(T_{F})$,
\begin{align*}
\langle \int_{\Omega}\langle f,G(\omega)\rangle F(\omega)d\mu (\omega),e\rangle &= \int_{\Omega}\langle f,\bar{\alpha(\omega)}h\rangle\langle F(\omega),e\rangle d\mu (\omega)\\
&=\int_{\Omega}\langle f,h\rangle \alpha(\omega) \langle F(\omega),e\rangle d\mu (\omega)\\
&=\langle f,h\rangle \langle \alpha,\{\langle e,F(\omega)\rangle \}_{\omega \in \Omega}\rangle \\
&=\langle f,h\rangle \langle \alpha,T_{F}e\rangle = 0 \quad \forall e,f \in H.
\end{align*}
Then $\int_{\Omega}\langle f,G(\omega)\rangle F(\omega)d\mu (\omega)=0 \quad \forall f\in H$.\\
Which give:
\begin{align*}
\int_{\Omega}\langle f,Q(\omega)\rangle F(\omega)d\mu (\omega)&=\int_{\Omega}\langle f,\tilde{F}(\omega)\rangle F(\omega)d\mu (\omega) + \int_{\Omega}\langle f,G(\omega)\rangle F(\omega)d\mu (\omega)\\
&=\int_{\Omega}\langle f,\tilde{F}(\omega)\rangle F(\omega)d\mu (\omega)=Kf.
\end{align*}
Since $\alpha\neq 0$, there exists $w_{0} \in \Omega$ such that $\alpha(w_{0})\neq 0$, and thus $G(w_{0})\neq 0$. A simple calculation gives $(\frac{\alpha(w_{0})}{|\alpha(w_{0})|^{2}})S(G(w_{0}))=\alpha$.\\
	Hence $Q$ is a dual continuous $K$-Bessel sequence of $F$ and its different from $\tilde{F}$ wich is a contradiction.
\end{proof}
\begin{theorem}\label{t2}
Suppose that $K\in B(H)$ has closed range and $F$ is a parseval continuous $K$-frame of $H$ then the following results hold:
\begin{itemize}
	\item [1 -] $F$ is continuous $L^{2}$-independent if and only if $\tilde{F}$ is continuous $L^{2}$-independent.
	 \item [2 -] If $F$ admits a unique dual continuous $K$- Bessel sequence then $\tilde{F}$ admits a unique dual continuous $K^{\ast}$-Bessel sequence. 
\end{itemize}
\end{theorem}
\begin{proof}
(1) One has
\begin{align*}
\int_{\Omega}\langle f,F(\omega)\rangle F(\omega)d\mu (\omega)&=T^{\ast}_{F}T_{F}f=KK^{\ast}f\\
&=\int_{\Omega}\langle K^{\ast}f,\tilde{F}(\omega)\rangle F(\omega)d\mu (\omega)\\
&=\int_{\Omega}\langle f,K\tilde{F}(\omega)\rangle F(\omega)d\mu (\omega) \qquad \forall f\in H.
\end{align*}
Hence, 
\begin{align*}
0&=\int_{\Omega}\langle f,F(\omega)\rangle F(\omega)d\mu (\omega)-\int_{\Omega}\langle f,K\tilde{F}(\omega)\rangle F(\omega)d\mu (\omega)\\
&=\int_{\Omega}\langle f,F(\omega)-K\tilde{F}(\omega)\rangle F(\omega)d\mu (\omega).
\end{align*}
Since, $F$ is continuous $L^{2}$-independent, it is follows that:
\begin{equation*}
\langle f,F(\omega)-K\tilde{F}(\omega)\rangle = 0 \qquad \forall f\in H, \quad \forall \omega \in \Omega.
\end{equation*} 
So $F=K\tilde{F}$.\\
Suppose now that $\int_{\Omega}c(\omega)\tilde{F}(\omega)d\mu (\omega)=0 \quad $ for some $c\in L^{2}(\Omega)$, then :
\begin{equation*}
0=K\int_{\Omega}c(\omega)\tilde{F}(\omega)d\mu (\omega)=\int_{\Omega}c(\omega)K\tilde{F}(\omega)d\mu (\omega)=\int_{\Omega}c(\omega)F(\omega)d\mu (\omega).
\end{equation*}
So, $c(\omega)=0 \qquad \forall \omega \in \Omega$, because $F$ is continuous $L^{2}(\Omega)$ independent.\\
For the converse, let $\int_{\Omega}c(\omega)F(\omega)d\mu (\omega)=0\quad $ for  $c \in L^{2}(\Omega)$, then: 
\begin{equation*}
0=K^{\dagger}\int_{\Omega}c(\omega)F(\omega)d\mu (\omega)=\int_{\Omega}c(\omega)K^{\dagger}F(\omega)d\mu (\omega)=\int_{\Omega}c(\omega)\tilde{F}(\omega)d\mu (\omega)
\end{equation*}
Then $c(\omega) = 0 \qquad \forall \omega \in \Omega$.\\
(2) Since, $F$ has a unique dual continuous $K$-Bessel, by theorem \ref{t1} we know that its analysis operator $T_{F}$ is surjective and  thus $T^{\ast}_{F}$ is injective, which implies that $F$ is continuous $L^{2}$-independent.\\
Hence, by (1), $\tilde{F}$ is also continuous $L^{2}$-independent, from wich we conclude that $\tilde{F}$ has a unique dual continuous $K^{\ast}$-Bessel sequence.
\end{proof}

\begin{theorem}\label{t4}
Suppose that $K\in B(H)$ has closed range and $F$ is a parseval continuous $K$-frame of $H$. Then for any $c\in L^{2}(\Omega)$ satisfying the equation:\\
$Kf=\int_{\Omega}c(\omega)F(\omega)d\mu (\omega)$, we have:
\begin{equation*}
\int_{\Omega}|c(\omega)|^{2}d\mu (\omega)=\int_{\Omega}|c(\omega)-\langle f,\tilde{F}(\omega)\rangle |^{2}d\mu (\omega) + \int_{\Omega}|\langle f,\tilde{F}(\omega)\rangle|^{2}d\mu (\omega).
\end{equation*}
\end{theorem}
\begin{proof}
	We have :
	\begin{align*}
	\int_{\Omega}(c(\omega)-\langle f,\tilde{F}(\omega)\rangle) \langle \tilde{F}(\omega),f\rangle d\mu (\omega) &=\int_{\Omega} \langle (c(\omega)-\langle f,\tilde{F}(\omega)\rangle)\tilde{F}(\omega),f\rangle d\mu (\omega)\\
	&=\langle \int_{\Omega} (c(\omega)-\langle f,\tilde{F}(\omega)\rangle)\tilde{F}(\omega),f\rangle d\mu (\omega)\\
	&=\langle K^{\dagger}\int_{\Omega} (c(\omega)-\langle f,\tilde{F}(\omega)\rangle)\tilde{F}(\omega),f\rangle d\mu (\omega)\\
	&=\langle K^{\dagger}(Kf-Kf),f\rangle =0 \qquad \forall f\in H.
	\end{align*}
	Therefore 
	\begin{align*}
	\int_{\Omega}|c(\omega)|^{2}d\mu (\omega)&=\int_{\Omega}c(\omega)\overline{c(\omega)}d\mu (\omega)\\
	&=\int_{\Omega}[(c(\omega)-\langle f,\tilde{F}(\omega)\rangle) +\langle \tilde{F}(\omega),f\rangle] \overline{[(c(\omega)-\langle f,\tilde{F}(\omega)\rangle) +\langle f,\tilde{F}(\omega)\rangle]}d\mu (\omega)\\
	&=\int_{\Omega}(((c(\omega)-\langle f,\tilde{F}(\omega)\rangle) \overline{(c(\omega)-\langle f,\tilde{F}(\omega)\rangle)} + (c(\omega)-\langle f,\tilde{F}(\omega)\rangle)\langle \tilde{F}(\omega),f\rangle  \\
	&+\langle f,\tilde{F}(\omega)\rangle \overline{(c(\omega)-\langle f,\tilde{F}(\omega)\rangle)} + \langle f,\tilde{F}(\omega)\rangle \langle \tilde{F}(\omega),f\rangle))d\mu (\omega)\\
	&=\int_{\Omega}(c(\omega)-\langle f,\tilde{F}(\omega)\rangle)\overline{(c(\omega)-\langle f,\tilde{F}(\omega)\rangle)}d\mu (\omega) + \int_{\Omega}\langle f,\tilde{F}(\omega)\rangle \langle \tilde{F}(\omega),f\rangle d\mu (\omega) \\
	&=\int_{\Omega}|c(\omega)-\langle f,\tilde{F}(\omega)\rangle|^{2} d\mu (\omega)+ \int_{\Omega}|\langle f,\tilde{F}(\omega)\rangle|^{2}d\mu (\omega).
	\end{align*}
\end{proof}
{\bf Conflict of interest}
On behalf of all authors, the corresponding author states that there is no conflict of interest.

{\bf Data availability statement:}  No data were used to support this study.

{\bf Funding statement:}  This study was not funded.
\bibliographystyle{amsplain}

\begin{thebibliography}{XX}
\bibitem{4} S. T. Ali, J. P. Antoine and J. P. Gazeau, continuous frames in Hilbert spaces, Annals of physics, 222 (1993), 1-37.
\bibitem{Tousi} A. A. Arefijamaal, R. A. Kamyabi Gol, R. Raisi Tousi and N. Tavallaei, A new approach
to continuous Riesz bases. J. Sci. Iran 24 (2013), 63-69.

\bibitem{10} O. Christensen, An Introduction to Frames and Riesz Bases, Birkhäuser, Boston(2003).
\bibitem{11} R. G. Douglas, On majorization, factorization and range inclusion of operators on Hilbert space, Proc. Amer. Math. Soc. 17, 413–415 (1966).
\bibitem{1} R. J. Duffin, A. C. Schaeffer, \emph{A class of nonharmonic fourier series}, Trans. Amer. Math. Soc. 72 (1952),
341-366.
\bibitem{2} I. Daubechies, A. Grossmann and Y. Meyer, \emph{Painless nonorthogonal expansions}, J. Math.
Phys. 27 (1986), 1271-1283.
\bibitem{3} D. Gabor, \emph{Theory of communications}, J. Elec. Eng. 93 (1946), 429-457.
\bibitem{5} A. Rahimi, A. Najati and Y. N. Deghan, Continuous frames in Hilbert spaces, Methods of Functional Analysis and Topology Vol. 12(2), (2006), 170-182.
\bibitem{c-k-frame} GH. Rahimlou, R. Ahmadi, M. A. Jafarizadeh and S. Nami, Continuous $K$-frames and their duals, arXiv: 1901.03803v1, (2019).
\bibitem{r1} M. Rossafi, S. Kabbaj, $\ast$-K-operator Frame for $End_{\mathcal{A}}^{\ast}(\mathcal{H})$,  Asian-Eur. J. Math. 13 (2020), 2050060.
\bibitem{r11} M. Rossafi, A. Bourouihiya , H. Labrigui and A. Touri, The duals of $\ast$-operator Frame for $End_{\mathcal{A}}^{\ast}(\mathcal{H})$, Asia Math. 4 (2020), 45-52.
\bibitem{r2} M. Rossafi, A. Touri, H. Labrigui and A. Akhlidj, Continuous $\ast$-K-G-Frame in Hilbert $C^{\ast}$-Modules, Journal of Function Spaces, (2019), Article ID 2426978.
\bibitem{r3} M. Rossafi, S. Kabbaj, Operator Frame for $End_{\mathcal{A}}^{\ast}(\mathcal{H})$, J. Linear Topol. Algebra, 8 (2019), 85-95.
\bibitem{r4} S. Kabbaj, M. Rossafi, $\ast$-operator Frame for $End_{\mathcal{A}}^{\ast}(\mathcal{H})$, Wavelet Linear Algebra, 5, (2) (2018), 1-13.
\bibitem{r5} M. Rossafi, S. Kabbaj, $\ast$-K-g-frames in Hilbert $\mathcal{A}$-modules, J. Linear Topol. Algebra, 7 (2018), 63-71.
\bibitem{r6} M. Rossafi, S. Kabbaj, $\ast$-g-frames in tensor products of Hilbert $C^{\ast}$-modules, Ann. Univ. Paedagog. Crac. Stud. Math. 17 (2018), 17-25.
\bibitem{r7} M. Rossafi, S. Kabbaj, K-operator Frame for $End_{\mathcal{A}}^{\ast}(\mathcal{H})$, Asia Math. 2 (2018), 52-60.
\bibitem{r10} M. Rossafi, S. Kabbaj, Frames and Operator Frames for $B(\mathcal{H})$, Asia Math. 2 (2018), 19-23.
\bibitem{r8} M. Rossafi, A. Akhlidj, Perturbation and Stability of Operator Frame for $End_{\mathcal{A}}^{\ast}(\mathcal{H})$,
Math-Recherche Appl. Vol. 16 (2017-2018), 65-81.
\bibitem{r9} M. Rossafi, S. Kabbaj, Generalized Frames for $B(\mathcal{H, K})$, Iran. J. Math. Sci. Inf. accepted.
\end{thebibliography}

\end{document}